\documentclass[11pt]{article}

\usepackage[T1]{fontenc}
\usepackage[utf8]{inputenc}
\usepackage{lmodern}
\usepackage{amsmath,amssymb,amsthm,mathtools}
\usepackage{bm}
\usepackage{enumitem}
\usepackage{geometry}
\usepackage{hyperref}
\usepackage{microtype}
\usepackage{array}
\usepackage[toc,page]{appendix}

\hypersetup{
    colorlinks=true,
    linkcolor=blue,
    citecolor=blue,
    urlcolor=blue
}

\newtheorem{theorem}{Theorem}
\newtheorem{lemma}{Lemma}
\newtheorem{remark}{Remark}

\newtheorem{proposition}{Proposition}

\title{An Exact Continuous Conductance Formulation of the Hamiltonian Path Problem}
\author{A. Sadovsky}
\date{}

\begin{document}

\maketitle

\begin{abstract}
The Hamiltonian Path Problem is formulated as a continuous minimization problem on conductances assigned to an Ohmic network associated with the given graph. The objective function is a sum of two penalty terms that collectively enforce a set of conditions sufficient for a subgraph of the original graph to be a Hamiltonian path. The objective function is nonconvex. The main result (Theorem~\ref{thm:main}) shows that, provided the graph has a Hamiltonian path from $h^{start}$ to $h^{end}$, a conductance configuration is a global minimizer of the objective if and only if it corresponds to a Hamiltonian path.
\end{abstract}

\tableofcontents

\section{Introduction}
\label{sec:introduction}

The Hamiltonian Path Problem (HPP)--the problem of deciding whether a given undirected graph contains a path that includes every node exactly once--is NP-complete \cite{Karp1972}. The problem arises naturally in scheduling, printed-circuit layout, genome assembly, and a range of combinatorial design tasks. The next four paragraphs briefly review related work.

Exact algorithms for HPP based on dynamic programming achieve time complexity $O(N^2 2^N)$ and space $O(N 2^N)$; see the Held-Karp method \cite{HeldKarp1962}. Practically efficient backtracking algorithms with pruning have been developed for moderate graph sizes. For large instances, heuristic and metaheuristic methods---including simulated annealing, genetic algorithms, and ant colony optimization---are widely used \cite{Applegate2006}.

A separate and increasingly active line of work reformulates graph problems as quadratic unconstrained binary optimization (QUBO) or Ising Hamiltonian minimization, targeting quantum annealers and gate-based quantum devices. Lucas~\cite{Lucas2014} gave systematic Ising-model encodings for a large class of NP problems including HPP, expressing the Hamiltonian path constraint through penalty terms that are quadratic in binary variables. These formulations have been implemented on D-Wave systems and small quantum processors, though the problem sizes accessible to current hardware remain modest.

Yet another category of approaches is inspired by neural networks and physics. Hopfield and Tank~\cite{HopfieldTank1985} proposed continuous-valued neural networks whose energy function encodes the travelling salesman problem, a relative of HPP, and showed empirically that gradient descent on this energy finds good solutions. Their approach was influential in establishing the idea that combinatorial structure can be encoded in continuous objectives, even though the energy function there has no provable exactness guarantee. More recent deep-learning and reinforcement-learning approaches learn heuristics for routing and sequencing problems \cite{VinyalsEtAl2015} but do not provide theoretical guarantees on solution quality.

A comprehensive survey of progress on the HPP and the closely related Hamiltonian cycle problem (where a cycle is sought that includes every node of the graph exactly once) can be found in \cite{gould2003advances}.

This paper proposes a formulation distinct from all of the above. The HPP is approached here by designing a theoretical electric network that includes the given graph and places variable conductances on its arcs. Namely, a connected undirected graph
\begin{equation}
\label{equation:original-graph}
G_{0} = (\mathcal N_0,\mathcal A_0)
\end{equation}
with node set $\mathcal N_0$ of size
\begin{equation}
\label{equation:node-set-size}
N_0 = |\mathcal N_0| \geq 3
\end{equation}
and arc set $\mathcal A_0$ is given. Hamiltonian paths in (\ref{equation:original-graph})  sought. The nodes in (\ref{equation:original-graph}) are treated as the nodes in an electric network, and the arcs as branches in the network. Two nodes,
$$
h^{start}, h^{end}  \in  \mathcal N_0,
$$
are chosen, by a method described below, as the intended endpoints of a Hamiltonian path.

The electric network interpretation is to think of an external source of electromotive force (e.m.f.) with the terminals connected to $h^{start}$ and $h^{end}$. The terminal connected to $h^{end}$ is grounded. The other will have a positive potential value, specified and explained below. 

The arcs, in their role of electric network branches, are furnished with adjustable conductances
$$
g_a  \in  [0,1], \quad a  \in  \mathcal A_0.
$$
A mapping $g$ that assigns to each arc $a  \in  \mathcal A_0$ a conductance $g_{a}  \in  [0, 1]$ will be called {\em a conductance configuration}.

This theoretical design results in each arc carrying a current which either is zero or has a direction that orients the arc. Therefore, the model involves assigning to an undirected arc $\{i,j\}  \in  \mathcal A_0$ one of the two directions $(i,j)$ or $(j,i)$ whenever the induced current through that arc is nonzero.

The goal is to find such a conductance configuration
\begin{equation}
\label{equation:conductance-configuration}
g = \left\{ g_a  \in  [0, 1] : a  \in  \mathcal A_0 \right\}
\end{equation}
that the resulting currents have positive unit value, $1$, along a sequence
$$
\left( h^{start} = v_{1}, v_{2} \right), \left(v_{2}, v_{3} \right), \ldots, \left( v_{N_0-1}, v_{N_0} = h^{end} \right)
$$
of directed arcs that determines a Hamiltonian path
\begin{equation}
\label{equation:HP-in-G_0}
h^{start} = v_{1} \to v_{2} \to \ldots \to v_{N_0-1} \to v_{N_0} = h^{end},
\end{equation}
while the currents on all the other arcs in $\mathcal{A}_{0}$ are zero.

Such conductance configurations will be sought as minimizers of an objective function, defined in section~\ref{section:total-objective}. This objective function will be a sum of terms that penalize conductance configurations for failing the key requirements for an HPP: local unit out-flow at all nodes but one (section~\ref{section:local-out-flow}) and absence of currents branching out of a node (section~\ref{section:branching-penalty}).

The main result of this paper, Theorem~\ref{thm:main} (section~\ref{sec:main-result}), states that if the graph has a Hamiltonian path, then every global minimizer $g_{*}$ of the objective determines a Hamiltonian path. The proof is constructive, and the argument is exact. The theorem characterizes the global minimizers but does not assert that finding them is polynomial-time. The objective $\widetilde{\mathcal{L}}$ is generally not convex.

The electric network formulation and the graph augmentation used to connect the e.m.f. to the graph are given in section~\ref{sec:HPP-as-linear-circuit-problem}. The potentials and currents resulting from the terminal boundary conditions and from Ohm's and Kirchhoff's laws are derived in section~\ref{section:potentials-and-currents}. The penalty terms are defined in section~\ref{section:penalty-terms}. The key results are stated and proved in section~\ref{sec:main-result}. 

\section{The electric network framework for seeking Hamiltonian paths}
\label{sec:HPP-as-linear-circuit-problem}

\subsection{Terminal-node construction}
\label{section:terminal-node-construction}

In the tasks of choosing the nodes $h^{start}, h^{end}$ in $\mathcal{N}_{0}$ and designing their connections to the terminals of a source of e.m.f., it is helpful to distinguish the four cases listed in Table~\ref{table:e.m.f.-terminal-design}.
\begin{table}[t]
\caption[]{Design for connecting the terminals of an e.m.f. source to the given graph.}
\begin{tabular}{p{1.6in}p{4.0in}}
{\bf case} & {\bf design}\\ \hline
$G_{0}$ has more than two nodes of degree $1$. & $G_{0}$ has no Hamiltonian path.\\ \hline
$G_{0}$ has exactly two nodes of degree $1$. & These two nodes are the only candidates (up to a swap) for the roles of $h^{start}$ and $h^{end}$. Set $v^{IN} = h^{start}$, $v^{OUT}= h^{end}$, and
$$
\mathcal N := \mathcal N_0, \qquad \mathcal A := \mathcal A_0.
$$
\\ \hline
$G_{0}$ has exactly one node of degree $1$. & This node should serve as $h^{start}$. Choose $h^{end}$ by any method. Set $v^{IN}=h^{start}$, introduce a new node $v^{OUT} \notin \mathcal N_0$ and the arc $\{h^{end}, v^{OUT}\}$, and define
$$
\mathcal N := \mathcal N_0 \cup \{v^{OUT}\}, \qquad \mathcal A := \mathcal A_0 \cup \left\{ \left\{ h^{end}, v^{OUT} \right\} \right\}.
$$
\\ \hline
$G_{0}$ has no nodes of degree $1$. & Choose $h^{start}$ and $h^{end}$ by any method, introduce two new nodes $v^{IN},v^{OUT} \notin \mathcal N_0$, and add two new arcs:
$$
\mathcal N := \mathcal N_0 \cup \{v^{IN},v^{OUT}\}, \qquad \mathcal A := \mathcal A_0 \cup \bigl\{\{v^{IN}, h^{start}\},\{h^{end}, v^{OUT}\}\bigr\}.
$$
\\ \hline
\end{tabular}
\label{table:e.m.f.-terminal-design}
\end{table}
Whichever of the cases listed in Table~\ref{table:e.m.f.-terminal-design} takes place, take the resulting, ``augmented,'' node set $\mathcal{N}$ and arc set $\mathcal{A}$ and denote the resulting graph by
\begin{equation}
\label{equation:augmented-graph}
G=(\mathcal N,\mathcal A).
\end{equation}
Note that always
\begin{equation}
\label{equation:vIN-and-vOUT-are-leaves}
\deg(v^{IN})=\deg(v^{OUT})=1.
\end{equation}
By the above construct, 
\begin{equation}
\label{equation:HP-bijection}
\left.
\begin{array}{c}
\mbox{\em A Hamiltonian path (\ref{equation:HP-in-G_0}) in $G_{0}$ corresponds to, and is a subpath of, }\\
\mbox{\em a unique Hamiltonian path from $v^{IN}$ to $v^{OUT}$ in $G$.}\\
\end{array}
\right\}
\end{equation}
Let
\begin{equation}
\label{equation:definition-of-N}
N:=|\mathcal N|.
\end{equation}

If arc $\{v^{IN},v^{OUT}\}$ is present in the original graph, i.e., $\{v^{IN},v^{OUT}\} \in \mathcal A_0$, then its conductance is fixed at $0$. Otherwise, the graph would contain an arc directly connecting the terminals of the voltage source, creating a shunt path outside the intended Hamiltonian structure. This does not alter the Hamiltonian-path problem, since a Hamiltonian path from $v^{IN}$ to $v^{OUT}$ cannot use this arc unless $N=2$, a possibility ruled out by (\ref{equation:node-set-size}).

\subsection{The space of conductance configurations}
\label{sec:conductance}

The set of all possible conductance configurations (\ref{equation:conductance-configuration}) is the cube $[0,1]^{|\mathcal A_0|}$. The conductances on the arcs $a  \in  \mathcal A \setminus \mathcal A_0$ (i.e., on whatever arcs, if any, that were added per Table~\ref{table:e.m.f.-terminal-design}) are not variables. They are constant parameters in the model, always set equal to $1$.

Denote by $\mathcal A_0[g=1] := \{a \in \mathcal A_0:g_a=1\}$ the set of all arcs in $\mathcal A_0$ that carry the maximal conductance of $1$. This set determines the subgraph $(\mathcal N_0,\mathcal A_0[g=1])$ of the original graph and, once the nodes $v^{IN}, v^{OUT}$ are identified and one or both of the arcs $\{v^{IN}, h^{start}\}$, $\{h^{end}, v^{OUT}\}$ added if needed, also the subgraph
\begin{equation}
\label{equation:full-conductance-subgraph-in-augmented}
\left( \mathcal N, \; \{a  \in  \mathcal A : g_a=1\} \right)
\end{equation}
of the augmented graph.

\subsection{Target set}

Define the subset
\begin{equation}
\label{equation:target-set}
\mathcal T_{HP} := \left\{ g \in \{0,1\}^{|\mathcal A_0|}: \mathcal A_0[g=1] \text{ is the arc set of a Hamiltonian path from } v^{IN} \text{ to } v^{OUT} \right\}
\end{equation}
of $[0,1]^{|\mathcal A_0|}$. Hamiltonian paths in the original graph will be sought as conductance configurations in $\mathcal T_{HP}$.

\begin{remark}
The elements of $\mathcal T_{HP}$ are vertices of the cube $[0,1]^{|\mathcal A_0|}$, i.e., binary conductance configurations. Establishing that a global minimizer of the objective $\widetilde{\mathcal{L}}$ over the full cube must be a vertex, and moreover a vertex in $\mathcal T_{HP}$, is the central result proved in section~\ref{sec:main-result}.
\end{remark}

\section{The potentials and currents derived from Ohm's and Kirchhoff's laws}
\label{section:potentials-and-currents}

\subsection{Voltage boundary conditions}
\label{section:voltage-boundary-conditions}

The external source of e.m.f.\ used here is an ideal {\em voltage source} \cite{Seshu1959linear}. No arc is added to the graph to house that voltage source. The potentials at its terminals will be chosen so that the potential difference between the nodes $v^{IN}$ and $v^{OUT}$ is $(N-1)$. Namely, set these potentials to:
\begin{equation}
\label{equation:potential-boundary-values}
\phi_{v^{IN}} := N-1, \qquad \phi_{v^{OUT}} := 0.
\end{equation}
The reason for this choice is as follows. If a conductance configuration $g$ is such that the subgraph (\ref{equation:full-conductance-subgraph-in-augmented}) is a Hamiltonian path from $v^{IN}$ to $v^{OUT}$ in the augmented graph, then {\em this path always has $(N-1)$ arcs (see (\ref{equation:definition-of-N})), regardless of which of the cases listed in Table~\ref{table:e.m.f.-terminal-design} is at hand}. Thus, the path has total resistance $(N-1)$ and potential drop (\ref{equation:potential-boundary-values}), and consequently carries unit current in its every arc.

\subsection{Effective conductance configurations and the conductance-weighted Laplacian}

Throughout the rest of this section, let $g$ be a conductance configuration.

Define {\em the effective conductance configuration $\tilde{g}$ corresponding to $g$} by
$$
\tilde g_a = \begin{cases} g_a, & a \in \mathcal A_0, \\ 1, & a \in \mathcal A\setminus\mathcal A_0. \end{cases}
$$
Thus the augmentation arcs, when present, are assigned fixed conductance $1$.

The weighted Laplacian associated with $\tilde g$ on the augmented graph $G=(\mathcal N,\mathcal A)$ is defined by
$$
L_{ii}(g) = \sum_{j:\{i,j\} \in \mathcal A} \tilde g_{\{i,j\}}, \qquad L_{ij}(g) = L_{ji}(g) = -\tilde g_{\{i,j\}}, \quad i\neq j.
$$

Let $\mathcal A[g>0] := \mathcal A_0[g>0] \cup (\mathcal A\setminus\mathcal A_0)$ denote the set of all the arcs with positive conductances under configuration $g$. The subgraph
\begin{equation}
\label{equation:conducting-subgraph}
(\mathcal N,\mathcal A[g>0])
\end{equation}
of the augmented graph will be called {\em the conducting subgraph} associated with $g$.

\subsection{The potential vector and the directed currents}
\label{section:potential-vector}

A connected component of the conducting subgraph (\ref{equation:conducting-subgraph}) that contains neither of the nodes $v^{IN}, v^{OUT}$ has no source of e.m.f. A connected component that contains one of these two nodes but not the other is an open circuit. In either case, such a component has current zero on every arc. By convention, the potential $\phi_{i}(g)$ at every node in such a component will be taken to equal zero.

Now suppose that $v^{IN}$ and $v^{OUT}$ lie in the same connected component of the conducting subgraph (\ref{equation:conducting-subgraph}). Let $\mathcal{N}_{P} \subset \mathcal N$ denote the node set of that connected component: the subscript {\em P} stands for ``{\em p}owered by the voltage source.'' The potential vector $\phi(g) = (\phi_i(g))_{i \in \mathcal{N}_{P}}$ for this component will be the unique solution of the Dirichlet problem consisting of the Kirchhoff equations
\begin{equation}
\label{equation:voltage-potential-equation}
\sum_{j:\{i,j\} \in \mathcal A} \tilde g_{\{i,j\}} \bigl( \phi_i(g)-\phi_j(g) \bigr) = 0, \qquad i \in  \mathcal{N}_{P} \setminus \{v^{IN},v^{OUT}\}
\end{equation}
and the Dirichlet boundary conditions (\ref{equation:potential-boundary-values}).

For each undirected arc $\{i,j\} \in \mathcal A$, introduce the oriented arcs $(i,j)$ and $(j,i)$. Define {\em the directed current} by Ohm's law:
\begin{equation}
\label{equation:single-directed-current}
J_{(i,j)}(g) := \tilde g_{\{i,j\}} \bigl( \phi_i(g)-\phi_j(g) \bigr).
\end{equation}
An immediate consequence of the above is $J_{(i,j)}(g) = -J_{(j,i)}(g)$.

The set of all arcs carrying nonzero currents,
\begin{equation}
\label{equation:absolute-current-support}
\mathcal A[Y > 0] := \bigl\{ \{i,j\}: |J_{(i,j)}(g)|>0 \bigr\},
\end{equation}
determines a subgraph of $G$ called the {\em absolute current support}.

\section{Penalty terms}
\label{section:penalty-terms}

The notation $s^+ := \max\{s,0\}$ for the positive part of a real-valued function will be used below.

\subsection{Local out-flow penalty}
\label{section:local-out-flow}

For each node $i \in \mathcal N$, define the total squared positive outgoing current
\begin{equation}
\label{equation:outflow}
O_i(g) := \sum_{j:\{i,j\} \in \mathcal A}  \bigl(J_{(i,j)}(g)\bigr)^{+} 
\end{equation}
A Hamiltonian path from $v^{IN}$ to $v^{OUT}$ would require every node except $v^{OUT}$ to have $O_i(g) = 1$:
\begin{equation}
\label{equation:unit-out-flow-requirement}
O_i(g) = \begin{cases} 0, & i=v^{OUT} \quad \mbox{(see the top paragraph of section~\ref{section:voltage-boundary-conditions})},\\ 1, & i\neq v^{OUT}. \end{cases}
\end{equation}
To capture requirement (\ref{equation:unit-out-flow-requirement}) as a penalty, define {\em the out-flow penalty}
\begin{equation}
\label{equation:unit-flow-penalty}
\mathcal{L}_{\mathrm{outflow}}(g) := O_{v^{OUT}}(g)^{2} + \sum_{\stackrel{i \in \mathcal N}{i \neq v^{OUT}}} \left( O_i(g)-1 \right)^2.
\end{equation}

The considerations stated in the top paragraph of section~\ref{section:potential-vector} show that penalty term (\ref{equation:unit-flow-penalty}) also enforces connectivity of the absolute current support: if any node $i \neq v^{OUT}$ is not reached by any current, then $O_i(g) = 0 \neq 1$, contributing a positive term to $\mathcal{L}_{\mathrm{outflow}}(g)$.

\subsection{Branching penalty}
\label{section:branching-penalty}

The penalty term introduced in this section is to penalize the undesirable situation when positive currents flow out of a node on two or more directed arcs emanating from that node. Fix a node $i  \in  \mathcal N$ and denote its neighbors by $j_1,\ldots,j_{d_i}$. Define {\em the local branching penalty at node $i$} by
\begin{equation}
\label{equation:local-branching-penalty}
B_i(g) := \sum_{1\leq k<\ell\leq d_i} \, \bigl(J_{(i,j_k)}(g)\bigr)^{+}  \,  \bigl(J_{(i,j_\ell)}(g)\bigr)^{+} .
\end{equation}
Since every summand is nonnegative, $B_i(g)\geq 0$, with equality holding if and only if at most one outgoing arc from $i$ carries positive current. Define {\em the global branching penalty} by summing the local ones (\ref{equation:local-branching-penalty}) over all the nodes:
\begin{equation}
\label{equation:branching-penalty}
\mathcal{L}_{\mathrm{branch}}(g) := \sum_{i \in \mathcal N} B_i(g).
\end{equation}

\subsection{Total objective}
\label{section:total-objective}

Define the total objective
\begin{equation}
\label{equation:total-objective}
\widetilde{\mathcal{L}}(g) := \alpha_1 \, \mathcal{L}_{\mathrm{outflow}}(g) + \alpha_2 \, \mathcal{L}_{\mathrm{branch}}(g), \quad \alpha_{k} > 0, \; k=1,2.
\end{equation}
Because $(x^{+})^2 = 0 \Leftrightarrow x^{+} = 0$ for all real $x$, the zero sets of $O_i$ and $\mathcal{L}_{\mathrm{branch}}$ are the same whether the positive parts are squared or not. Theorem~\ref{thm:main} therefore holds for this objective. The use of $(J^+)^2$ rather than $J^+$ makes the objective continuously differentiable at configurations where arc currents pass through zero, as the derivative of $(x^{+})^2$ with respect to $x$ is $2x^+$, which is continuous at $x=0$.

\section{Theoretical results}
\label{sec:main-result}

In what follows, the end of a proof is marked with the symbol $\square$ (notation due to P. Halmos). The penalties introduced in section~\ref{section:penalty-terms} are intended together to force the absolute current support~(\ref{equation:absolute-current-support}) to become a Hamiltonian path from $v^{IN}$ to $v^{OUT}$.

\begin{proposition}[Ideal zero-penalty structure]
\label{prop:zero-penalty-structure}
The absolute current support (\ref{equation:absolute-current-support}) corresponding to a conductance configuration $g$ satisfying
$$
\mathcal{L}_{\mathrm{outflow}}(g) = \mathcal{L}_{\mathrm{branch}}(g) = 0
$$
is a Hamiltonian path from $v^{IN}$ to $v^{OUT}$.
\end{proposition}

\begin{proof}

The proof will proceed by first showing that the absolute current support (\ref{equation:absolute-current-support}) is acyclic, then that it is connected and that it is a spanning tree of the augmented graph $G$, and then that its only leaves are $v^{IN}$ and $v^{OUT}$ and that all the other nodes each have degree $2$. The desired result follows.

{\bf Proof of acyclicity.}
By Ohm's law (\ref{equation:single-directed-current}), $J_{(i,j)}~>~0$ implies
$\phi_i~>~\phi_j$. Hence a directed cycle of positive current would force
the potential to decrease around a closed loop, contradicting
single-valuedness of $\phi$. Thus the absolute current support is acyclic.

{\bf Proof of connectivity.} 
Since
$$
\mathcal L_{\mathrm{outflow}}(g)=0,
$$
condition (\ref{equation:unit-out-flow-requirement}) gives
$$
O_i(g)=1 \quad \text{for } i\neq v^{OUT},
\qquad
O_{v^{OUT}}(g)=0.
$$
These equalities, in turn, imply that every node except $v^{OUT}$ has a total outgoing current of $1$.  Condition $\mathcal L_{\mathrm{branch}}(g)=0$ implies that, for all nodes, current flows out of a node through at most one arc. Hence every node except $v^{OUT}$ has exactly one outgoing positive-current arc.

Starting at any node, construct a path from it recursively, node by node, as follows:  
\begin{enumerate}
\item If the current node has no outgoing positive current, stop: this is the last node of the path.
\label{step-1}
\item If the current node has an outgoing positive current, then this current is carried by a unique outgoing directed arc. Follow the arc to the next node, relabel the next node as the current node, and go to step \ref{step-1}.
\end{enumerate}
By acyclicity, this path has no repeated nodes and must terminate at a node with no outgoing positive current. The only such node is $v^{OUT}$. Therefore from every node distinct from $v^{OUT}$ there is a unique path to  $v^{OUT}$. This proves connectivity.

{\bf The tree is a spanning tree.} 
Since the absolute current support includes all $N$ nodes of the augmented graph, it is a spanning tree. It remains to examine the degrees of its nodes.

{\bf Nodes $v^{IN}$ and $v^{OUT}$ are the leaves, while all the other nodes each have degree $2$.}
By the construction in section~\ref{section:voltage-boundary-conditions}, $v^{IN}$ has no incoming positive current. At every node $i$ distinct from $v^{IN}$ and $v^{OUT}$, Kirchhoff's current law gives
$$
I_i(g)=O_i(g)=1.
$$
Thus every node except $v^{IN}$ and $v^{OUT}$ receives a total incoming current of $1$ and, therefore, has at least one incoming positive-current arc. Since the tree has $N-1$ arcs, every node except $v^{IN}$ has exactly one incoming positive-current arc.

It follows that $v^{IN}$ and $v^{OUT}$ are the only leaves of the tree and that all the other nodes each have degree $2$. Therefore the absolute current support is a path from $v^{IN}$ to $v^{OUT}$ that includes every node exactly once. Thus the absolute current support is a Hamiltonian path from $v^{IN}$ to $v^{OUT}$.
\end{proof}

\begin{lemma}
\label{lemma:existence-of-HP-implies-L-has-a-zero}
If the given graph $G_{0}$ (see (\ref{equation:original-graph})) has a Hamiltonian path from $h^{start}  \in  \mathcal{N}_{0}$ to $h^{end}  \in  \mathcal{N}_{0}$, then there exists a conductance configuration $g^{*}  \in  [0, 1]^{|\mathcal{A}_{0}|}$ at which the total objective (\ref{equation:total-objective}) vanishes:
\begin{equation}
\label{equation:total-objective=0-on-g*}
\widetilde{\mathcal{L}}(g^{*}) = 0.
\end{equation}
\end{lemma}

\begin{proof}
Suppose $G_{0}$ has a Hamiltonian path (\ref{equation:HP-in-G_0}) and after the augmentation described in Table \ref{table:e.m.f.-terminal-design} this path extends (see (\ref{equation:HP-bijection})) to the Hamiltonian path
\begin{equation}
\label{equation:given-HP}
v^{IN} = v_{1} \to v_{2} \to \ldots \to v_{N_0-1} \to v_{N} = v^{OUT}.
\end{equation}
Construct conductance configuration $g^{*}$ as follows. Assign conductance $1$ on each of the arcs $\{v_{k},v_{k+1}\}$ in (\ref{equation:given-HP}) and conductance $0$ on the rest of the arcs.

The conducting subgraph (\ref{equation:conducting-subgraph}) with $g = g^{*}$ is path (\ref{equation:given-HP}). It has a total of $N-1$ arcs, each of conductance $1$. The total resistance between $v^{IN}$ and $v^{OUT}$ is therefore $N-1$. Since the applied potential drop is also $N-1$, Ohm's law gives current $1$ through each arc of path (\ref{equation:given-HP}), directed from $v^{IN}$ to $v^{OUT}$, and current $0$ on every arc assigned conductance $0$. The absolute current support at $g^*$ is therefore (\ref{equation:given-HP}).

Direct computation shows that $\mathcal{L}_{\mathrm{outflow}}$ and $\mathcal{L}_{\mathrm{branch}}$ both vanish at $g^*$. Equality (\ref{equation:total-objective=0-on-g*}) follows.
\end{proof}

\begin{theorem}
\label{thm:main}
If the given graph $G_{0}$ has a Hamiltonian path from node $h^{start}$ to node $h^{end}$, then every global minimizer of the total objective (\ref{equation:total-objective}) over $\mathcal G_{0} = [0,1]^{|\mathcal A_0|}$ belongs to $\mathcal T_{HP}$.
\end{theorem}

\begin{proof}
Each summand in (\ref{equation:total-objective}) is nonnegative by construction, so the total penalty $\widetilde{\mathcal L}$ is a nonnegative function.

Suppose the augmentation $G$ of $G_{0}$ (see section \ref{section:terminal-node-construction}) has a Hamiltonian path (\ref{equation:given-HP}). By Lemma~\ref{lemma:existence-of-HP-implies-L-has-a-zero}, there exists a conductance configuration $g^* \in \mathcal T_{HP}$ such that $\widetilde{\mathcal L}(g^*)=0$. It follows that $0$ is not only a lower bound of $\widetilde{\mathcal L}$ but is actually attained, hence is the global minimum:
$$
\min_{g \in \mathcal G_0} \widetilde{\mathcal L} = 0.
$$
Now let $\hat g$ be any global minimizer of $\widetilde{\mathcal L}$:
\begin{equation}
\label{equation:g-hat-global-minimizer}
\widetilde{\mathcal L}(\hat g)=0.
\end{equation}
Since every coefficient $\alpha_k$ in (\ref{equation:total-objective}) is strictly positive and every penalty term is nonnegative, condition (\ref{equation:g-hat-global-minimizer}) implies
$$
\mathcal L_{\mathrm{outflow}}(\hat g) = \mathcal L_{\mathrm{branch}}(\hat g) = 0.
$$
By Proposition~\ref{prop:zero-penalty-structure}, the absolute current support (\ref{equation:absolute-current-support}) corresponding to $g = \hat g$ is a Hamiltonian path from $v^{IN}$ to $v^{OUT}$. We now determine the conductance values assigned by $\hat{g}$ to the arcs of this path.

Since $\mathcal{L}_{\mathrm{branch}}(\hat{g})=0$, outgoing current does not branch out at any node, so the same unit amount of outgoing current from $v^{IN}$ is carried by every arc of the Hamiltonian path in sequence. Each such arc $\{i, j\}$ therefore carries current $1$, and Ohm's law gives
$$
1 = \hat g_{\{i,j\}} \,|\phi_i-\phi_j|,
$$
which implies
\begin{equation}
\label{equation:voltage-drop-on-HP-arc}
|\phi_i-\phi_j| = \frac{1}{\hat g_{\{i,j\}}} \geq 1,
\end{equation}
since $0<\hat g_{\{i,j\}}\leq1$. Hamiltonian path (\ref{equation:given-HP}) has a total of $N-1$ arcs, while the total potential drop between $v^{IN}$ and $v^{OUT}$ equals $N-1$. Since each arc contributes a potential drop of at least $1$ by (\ref{equation:voltage-drop-on-HP-arc}), that contribution must be exactly $1$. Thus every arc of Hamiltonian path (\ref{equation:given-HP}) has conductance $1$.

It follows that all nodes of the path have distinct potentials. Now let $a=\{i,j\} \in \mathcal A_0$ be an arc not on the Hamiltonian path. If it was the case that $\hat g_a>0$, then Ohm's law would produce nonzero current on $a$, since the endpoints of $a$ have distinct potentials. This contradicts the fact that the Hamiltonian path is the absolute current support. Therefore every arc not on the Hamiltonian path has conductance $0$, and $\hat g \in \mathcal T_{HP}$.
\end{proof}

\section{Discussion}
\label{sec:discussion}

Theorem~\ref{thm:main} shows that, if the given graph contains a Hamiltonian path from
$
h^{start}
$
to
$
h^{end},
$
then the global minimizers of the objective
$
\widetilde{\mathcal L}
$
are precisely the conductance configurations corresponding to Hamiltonian paths. The proof uses only elementary properties of electric networks together with the structure imposed by the penalty terms.

The result should not be interpreted as a polynomial-time algorithm for the Hamiltonian path problem. The objective
$
\widetilde{\mathcal L}
$
is generally nonconvex and may possess many local minima. The numerical experiments show that straightforward projected gradient descent often fails to locate the global minimum even for graphs of moderate size. Thus the present work addresses the question of {\em exact characterization} rather than that of {\em efficient computation}. The formulation presented here, as a minimization of an objective on the unit cube in a Euclidean space, embeds the originally discrete HPP in a framework of analysis and geometry of continuous and functions with certain smoothness properties in Euclidean vector spaces. This allows, in particular, for a study of the gradient and Hessian of the objective function or of its modifications and for a geometric interpretation of some combinatorial properties of a Hamiltonian path, such as the statement that the target set (\ref{equation:target-set}) is contained in the intersection of cube $[0, 1]^{|\mathcal{A}_{0}|}$ and hyperplane
$$
\sum_{a \in \mathcal{A}} g_{a} = N-1.
$$

\bibliographystyle{plain}
\bibliography{bibliography-for-switched-systems-HPP}

\end{document}